\documentclass[a4paper]{article}
\usepackage{latexsym}
\usepackage{tikz,epsf}
\usepackage{tikz,amsmath}

\usepackage{amsmath,amssymb,amsfonts}
\usepackage{amsfonts}
\usepackage{amscd}
\tikzset{
    partial ellipse/.style args={#1:#2:#3}{
    insert path={+ (#1:#3) arc (#1:#2:#3)}
    }
}

\newtheorem{prethm}{{\bf Theorem}}

\newenvironment{thm}{\begin{prethm}{\hspace{-0.5
               em}{\bf.}}}{\end{prethm}}

\newtheorem{prepro}[prethm]{Proposition}

\newtheorem{prelem}[prethm]{Lemma}
\newenvironment{lem}{\begin{prelem}{\hspace{-0.5
               em}{\bf.}}}{\end{prelem}}

\newtheorem{precor}[prethm]{Corollary}
\newenvironment{cor}{\begin{precor}{\hspace{-0.5
               em}{\bf.}}}{\end{precor}}

\newtheorem{preobs}[prethm]{Observation}

\newtheorem{preremark}{{\bf Remark}}

\newtheorem{preexample}{{\bf Example}}

\newenvironment{example}{\begin{preexample}\em{\hspace{-0.5
               em}{\bf.}}}{\end{preexample}}

\newtheorem{preproblem}{{\bf problem}}

\newtheorem{preproof}{{\bf Proof.}}

\newenvironment{proof}[1]{\begin{preproof}{\rm
               #1}\hfill{$\Box$}}{\end{preproof}}

\renewcommand{\thefootnote}

\newcommand{\QEDmark}{\mbox{\textsc{qed}}}
\newcommand{\proofStarter}[1]{\textsc{#1} }

\newcommand{\vertex}{\node[vertex]}

\tikzstyle{vertex}=[circle, draw, inner sep=0pt, minimum size=4pt]
\tikzstyle{vertexa}=[circle, draw, inner sep=0pt, minimum size=2pt]
\begin{document}

\date{}
\title{\bf The multiplicity of the Laplacian eigenvalue $2$ in some bicyclic graphs}

\author{{
{\small Masoumeh Farkhondeh}$^{a}$, \small Mohammad Habibi$^{a}$,
\small Doost Ali Mojdeh$^{b}$, \small Yongsheng Rao$^{c}$}
 \\{\small $^{a}$Department of Mathematics,}
 \\{\small  Tafresh University, Tafresh, 39518-79611, Iran,}
 \\{\small $^{b}$Department of Mathematics,}
 \\{\small University of Mazandaran, Babolsar, 47416-95447, Iran}
\\{\small $^{c}$Guangzhou University, Guangzhou, 510006, China}
 \\
{\tt $^{a}$mfarkhondeh81@gmail.com},\\ {\tt  $^{a}$
habibi.mohammad2@gmail.com},\\ {\tt  $^{b}$damojdeh@umz.ac.ir},\\
{\tt  $^{c}$rysheng@gzhu.edu.cn}}

\date{}
\maketitle
\begin{abstract}
The Laplacian matrix of a graph $G$ is denoted by $L(G)=D(G)-A(G)$,
where $D(G)=diag(d(v_{1}),\ldots , d(v_{n}))$ is a diagonal matrix
and $A(G)$ is the adjacency matrix of $G$. Let $G_1$ and $G_2$ be
two graphs. A one-edge connection of two graphs $G_1$ and $G_2$ is a
graph $G=G_1\odot_{uv} G_2$ with $V(G)=V(G_1)\cup V(G_2)$ and $E(G)=
E(G_1)\cup E(G_2)\cup \{e=uv\}$, where $u\in V(G_1)$ and $v\in
V(G_2)$. We investigate the multiplicity of the Laplacian eigenvalue
$2$ of $G_1\odot_{uv} G_2$, while the unicyclic graphs $G_1$ and
$G_2$ have $2$ among their Laplacian eigenvalues, by using their
Laplacian characteristic polynomials. Some structural conditions
ensuring the presence of the existence $2$ in the $G=G_1\odot_{uv} G_2$
 where both $G_1$ and $G_2$ have $2$ as Laplacian eigenvalue,  have been
investigated, while, here we study the
existence Laplacian eigenvalue $2$ in $G=G_1\odot_{uv} G_2$ where at most one of
$G_1$ or $G_2$ has $2$ as Laplacian eigenvalue.

\end{abstract}



\section{Introduction}
\indent All graphs in this paper are finite and undirected with no
loops or multiple edges. Let $G$ be a graph with $n$ vertices. The
vertex set and the edge set of $G$ are denoted by $V(G)$ and $E(G)$,
respectively. The \textit{Laplacian matrix} of $G$ is
$L(G)=D(G)-A(G)$, where $D(G)=diag(d(v_{1}),\ldots , d(v_{n}))$ is a
diagonal matrix and $d(v)$ denotes the degree of the vertex $v$ in
$G$ and $A(G)$ is the adjacency matrix of $G$. We shall use the
notation $\lambda_k(G)$ to denote the $k^{\mathrm{th}}$ Laplacian
eigenvalue of the graph $G$ and we assume that $\lambda_{1}(G)\geq
\cdots \geq \lambda_{n}(G)=0$. Also, the multiplicity of the
eigenvalue $\lambda$ of $L(G)$ is denoted by $m_{G}(\lambda)$. A
vertex of degree one is called a \textit{leaf} vertex and a vertex
is said \textit{quasi leaf} (support vertex) if it is incident to a
leaf vertex. Connected graphs in which the number of edges equals
the number of vertices are called \textit{unicyclic graphs}.
Therefore, a unicyclic graph is either a cycle or a cycle with some
attached trees. Let $\mathfrak{U}_{n,g}$ be the set of all unicyclic
graphs of order $n$ with girth $g$. Throughout this paper, we
suppose that the vertices of the cycle $C_g$ are labeled by
$v_1,...,v_g$, ordered in a natural way around $C_g$, say in the
clockwise direction. A \textit{rooted tree} is a tree in which one
vertex has been designated the root. Furthermore, assume that
$T_{i}$ is a rooted tree of order $n_{i}\geq 1$ attached to
$v_{i}\in V(T_{i})\cap V(C_{g})$, where $\sum_{i=1}^{g}n_{i}=n$.
This unicyclic graph is denoted by $C(T_{1},...,T_{g})$. The
\textit{sun graph} of order $2n$ is a cycle $C_n$ with an edge
terminating in a leaf vertex attached to each vertex that is the
corona of $C_n\circ K_1$. A \textit{broken sun graph} is a unicyclic
subgraph of a sun graph, so one can assume a sun graph is a broken
sun graph too. A \textit{one-edge connection} of two graphs $G_1$
and $G_2$ is a graph $G=G_1\odot_{uv} G_2$ with $V(G)=V(G_1)\cup
V(G_2)$ and $E(G)= E(G_1)\cup E(G_2)\cup \{e=uv\}$, where $u\in
V(G_1)$ and $v\in V(G_2)$. We use the notation $u\sim v$ when
$e=uv\in E(G)$.

\indent By \cite [Theorem~13]{mieghem}, due to Kelmans and
Chelnokov, the Laplacian coefficient, $\xi_{n-k}$, can be expressed
in terms of subtree structures of $G$, for $0\leq k\leq n$. Suppose
that $F$ is a spanning forest of $G$ with components $T_{i}$ of
order $n_{i}$, and $\gamma(F)=\Pi_{i=1}^{k}n_{i}$. The
\textit{Laplacian characteristic polynomial} of $G$ is denoted by $
L_{G}(\lambda)= det(\lambda I-L(G)) =
\Sigma_{i=0}^{n}(-1)^{i}\xi_{i}\lambda^{n-i}$. If $M$ is a square
matrix, then the determinant of $M$ is denoted by $|M|$ and the
\textit{minor} of the entry in the $i^{\mathrm{th}}$ row and
$j^{\mathrm{th}}$ column is the determinant of the submatrix formed
by deleting the $i^{\mathrm{th}}$ row and $j^{\mathrm{th}}$ column.
This number is often denoted by $M_{i,j}$. A square matrix is
non-singular if its determinant is non-zero or it has an inverse.
\par
Let $G$ be a graph with $n$ vertices. It is convenient to adopt the
following terminology from \cite{fiedler}: for a vector
$X=(x_{1},\ldots,x_{n})^t\in \mathbb{R}^{n}$, we say $X$ gives a
valuation of the vertex of $V$, and with each vertex $v_{i}$ of $V$,
we associate the number $x_{i}$, which is the value of the vertex
$v_{i}$, that is $x(v_{i})=x_{i}$. Then $\lambda$ is an eigenvalue
of $L(G)$ with the corresponding eigenvector
$X=(x_{1},\ldots,x_{n})$ if and only if $X\neq 0$ and
\begin{equation}\label{one}
(d(v_{i})-\lambda)x_{i}=\sum_{v_{j}\in N(v_{i})} x_{j}\,\,
\mathrm{for}\,\,\mathrm{all}\,\, i=1,\ldots,n.
\end{equation}

In this article, we would like to study the eigenvalue $2$ in
bicyclic graphs with just $2$ cycles. The main Theorem is about the
multiplicity of the Laplacian eigenvalue $2$ of $G_1\odot_{uv} G_2$,
while the unicyclic graphs $G_1$ and $G_2$ have $2$ among their
Laplacian eigenvalues, by using their Laplacian characteristic
polynomials. By addition, we investigate the existence Laplacian
eigenvalue $2$ in $G=G_1\odot_{uv} G_2$ where at most one of $G_1$
or $G_2$ has $2$ as Laplacian eigenvalue, by using Equation
(\ref{one}).
\section{Main results}

 We start this section by studying the multiplicity
of the Laplacian eigenvalue $2$ of $G_1\odot_{uv} G_2$, where the
unicyclic graphs $G_1$ and $G_2$ have $2$ among their Laplacian
eigenvalues. For this, we need the following results.


\begin{lem}\label{determin}\emph{\cite[Lemma~2.2]{cvetkovic et
al.}} If $M$ is a non-singular square matrix, then
\begin{equation}\label{det}
det\left(%
\begin{array}{cc}
  M & N \\
  P & Q \\
\end{array}%
\right)=\begin{vmatrix}
M & N \\
P & Q
\end{vmatrix}=\mid M\mid.\mid Q-PM^{-1}N\mid
\end{equation}
\end{lem}
We use the notation $M_u$ when in $M_{i,i}, i$ is the row and column
corresponding to the vertex $u$ in a graph.
\begin{lem}\label{character}\emph{\cite[Lemma~8]{guo}}
Let $G_1$ and $G_2$ be two graphs of order $n$ and $m$,
respectively. Then the Laplacian characteristic polynomial of
$G=G_1\odot_{uv} G_2$ is
\begin{equation}\label{laplac}
L_{G}(\lambda)=M.Q-M.Q_v-Q.M_u
\end{equation}
 where $M=\lambda I-L(G_1)$ and $Q=\lambda I-L(G_2)$.
\end{lem}


\begin{cor}
Let $G$ be a unicyclic graph on $n$ vertices. If $\lambda_1,\ldots
,\lambda_n$ are the Laplacian eigenvalues of $G$, then
$\lambda_1,\ldots ,\lambda_n$ are the Laplacian eigenvalues of $G'=
G\odot_{uv} G$.
\end{cor}

\begin{proof}{
  Suppose $u_1,\ldots
,u_n$ and $v_1,\ldots ,v_n$ are the vertices of $G$ and the copy of
$G$, respectively. Let $G'=G\odot_{u_iv_j} G$ and $M =L_G(\lambda)$.
Therefore $L_{G'}(\lambda)=M^2-M.M_{v_j}-M.M_{u_i}$, by Equation
(\ref{laplac}). So $L_{G'}(\lambda)=M(M-M_{u_i}-M_{v_j})$ and the
result follows.}
\end{proof}
 In \cite{ mojdeh et al.}, we have shown the upper bound for
$m_G(\lambda)$, where $\lambda> 1$.

\begin{lem}\label{eqibi}\emph{\cite [Lemma~2]{mojdeh et al.}}
Let $G$ be a bicyclic graph and $\lambda> 1$ be an integral
eigenvalue of $L(G)$. It holds that $m_{G}(\lambda)\leq 3$.
\end{lem}
Now, we identified some bicyclic graphs that having $2$ among their
Laplacian eigenvalues with multiplicity $3$.

 \begin{thm}\label{multi}\textbf{\emph{[Main Theorem]}} Let $G_1$
and $G_2$ be unicyclic graphs containing a perfect matching with
$m_{G_1}(2)= m_{G_2}(2)=2$. It holds that $m_{G_1\odot_{uv}
G_2}(2)=3$ where  $u\in C_{g_1}$ and $v\in C_{g_2}$.
\end{thm}

\begin{proof}{
If $m_{G_1}(2)= m_{G_2}(2)=2$, then there are two situations.
\begin{enumerate}

    \item If $G_1=C_m$ and $G_2=C_n$ then $m\equiv n\equiv 0(\mathrm{mod}\,4)$, by
\cite [Theorem~12]{akbari et al.}.

    \item  If $G_1=C(T^{(1)}_1,\ldots
,T^{(1)}_{g_1})$ and $G_2=C(T^{(2)}_1,\ldots ,T^{(2)}_{g_2})$ such
that $\sum^{k=g_1}_{k=1}|V(T^{(1)}_k)|=m$,
$\sum^{k=g_2}_{k=1}|V(T^{(2)}_k)|=n$ and there exists at least one
$i$  (or $j$) so that $|V(T^{(1)}_i)|\geq 3$  (or
$|V(T^{(2)}_j)|\geq 3)$. Let
$$s_i=|\{T^{(i)}_k:
|V(T^{(i)}_k)|  \,\,\mbox{is}\,\, \mbox{odd}; 1\leq k\leq g_i\}| \,\
; \,\ i=1,2,$$ so $s_i=g_i$ and $g_i\equiv 0\ (\mathrm{mod}\,4)$,
for $i=1,2$, by \cite [Theorem~13]{akbari et al.}.
\end{enumerate}
In addition, $G=G_1\odot{uv} G_2$ is a bicyclic graph so $m_G(2)\leq
3$, by Lemma \ref{eqibi}. On the other hand,
$$L_{G_1}(\lambda)=|M|=(\lambda-2)^2 f(\lambda)\quad ,\quad
L_{G_2}(\lambda)=|Q|=(\lambda-2)^2 g(\lambda).$$ According to
Equation (\ref{laplac}), it is enough to show that $M_u$ and $Q_v$
have $\lambda-2$ as a factor.

\textbf{Case $1$.} Let $G_1=C_m$, $G_2=C_n$ and $m\equiv n\equiv
0(\mathrm{mod}\,4)$.

We use induction on $m$. If $G_1=C_4$ then the result follows and so
the induction basis holds.
$$M_u=\begin{vmatrix}
  \lambda-2 & 1 & 0 \\
  1 & \lambda-2 & 1 \\
  0 & 1 & \lambda-2 \\
\end{vmatrix}=(\lambda-2)(\lambda^2-4\lambda+2).$$
Suppose in the cyclic graph $C_{m-4}$, $M_u(C_{m-4})$ with
$m-4\equiv 0(\mathrm{mod}\,4)$ has $\lambda-2$ as a factor.

\noindent Let $G_1=C_m$ and without loss of generality, assume that
$u=u_m$. So
$$M_u= \begin{vmatrix}
   M' & N \\
  N^T & Q' \\
\end{vmatrix}$$

\noindent such that
$$M'=\left(%
\begin{array}{cccc}
  \lambda-2 & 1 & 0 & 0 \\
  1 & \lambda-2 & 1 & 0 \\
  0 & 1 & \lambda-2 & 1 \\
  0 & 0 & 1 & \lambda-2 \\
\end{array}%
\right),$$

$$Q'=(q'_{ij})\quad ;\quad q'_{ij}=\begin{cases}
         \lambda-2, & \quad\,\, i=j; \\
      a_{ij}, & \quad\,\, i\neq j; \\
\end{cases} \qquad ; \qquad 1\leq i,j\leq m-5$$

and $$N_{4\times(m-5)}=(n_{ij})\quad ; \quad n_{ij}=
\begin{cases}
         1, & \quad\,\,i=4,j=1 ; \\
         0, & \quad\,\,o.w ; \\
\end{cases} \qquad; \qquad 1\leq i\leq 4 \quad and \quad 1\leq j\leq m-5.$$
$$$$
\noindent Therefore, $M_u=|M'||Q'-N^TM'^{-1}N|$, by Lemma \ref{det}.

$$$$

\noindent Also
$$[N^TM'^{-1}N]_{m-5\times m-5}=\left(%
\begin{array}{ccccc}
  \frac{M'_{4,4}}{|M'|} & 0 & \cdots & 0 & 0\\
  0 & \ddots & \ddots &  & 0 \\
  \vdots &\ddots  & 0 &\ddots &\vdots \\
  0 &  &\ddots & \ddots & 0 \\
  0 & 0 & \cdots & 0 & 0 \\
\end{array}%
\right).$$

\noindent Then $$Q'-N^TM'^{-1}N=(b'_{ij})\quad ;\quad
b'_{ij}=\begin{cases}
         \lambda-2- \frac{M'_{4,4}}{|M'|}, & \quad\,\, i=j=1; \\
         \lambda-2, & \quad\,\, i=j\neq 1; \\
      a_{ij}, & \quad\,\, i\neq j; \\
\end{cases} \qquad;\qquad 1\leq i,j\leq m-5.$$

\noindent Furthermore, $|Q'-N^TM'^{-1}N|=|Q'|+|H|$, such that

 $$H=(h_{ij})\quad;\quad h_{ij}=\begin{cases}
         -\frac{M'_{4,4}}{|M'|}, & \quad\,\, i=j=1; \\
        0, & \quad\,\, i=1\neq j; \\
         \lambda-2, & \quad\,\, i=j\neq 1; \\
      a_{ij}, & \quad\,\, i\neq 1,i\neq j; \\
\end{cases}\qquad;\qquad 1\leq i,j\leq m-5.$$

\noindent Consequently,
$$|Q'-N^TM'^{-1}N|=|Q'|-\frac{M'_{4,4}}{|M'|}Q'_{1,1}$$
$$\Rightarrow M_u=|M'||Q'|-M'_{4,4}Q'_{1,1}.$$
 On the other hand $$M'_{4,4}=\begin{vmatrix}
  \lambda-2 & 1 & 0\\
  1 & \lambda-2 & 1 \\
  0 & 1 & \lambda-2 \\
\end{vmatrix}=(\lambda-2)(\lambda^2-4\lambda+2).$$
Also, $|Q'_{(m-5)\times (m-5)}|$ according to the induction
hypothesis has $\lambda-2$ as a factor therefore, $M_u$ has
$\lambda-2$ as a factor. With a similar method, one can check that
$Q_v$ has $\lambda-2$ as a factor. Thus, $m_G(2)=3$ and we are done.

\textbf{Case $2$.} Let $G_1=C(T^{(1)}_1,\ldots ,T^{(1)}_{g_1})$ and
$G_2=C(T^{(2)}_1,\ldots ,T^{(2)}_{g_2})$ so $s_i=g_i$ and $g_i\equiv
0\ (\mathrm{mod}\,4)$, for $i=1,2$. In addition, there exists some
$i$ ($j$), such that $|V(T^{(1)}_i)|\geq 3$ ($|V(T^{(2)}_j)|\geq
3$). Let $u'\in V(T^{(1)}_i)$ and
 $d(u',u_i)=max_{x\in V(T^{(1)}_i)}d(x,u_i)$, where $u_i$ is the root
 of $T^{(1)}_i$. Since $G_1$ has a perfect matching, $u'$
is a pendent vertex and its neighbor, say $u''$, has degree $2$.
Thus, $G =(G\setminus\{u',u''\})\odot S_2$. Using \cite
[Theorem~2.5]{grone et al.} we obtain $m_G(2)
=m_{G\setminus\{u',u''\}}(2)$. So by repeating this method in the
graph $G\setminus\{u',u''\}$ and omitting all leaf vertices and
quasi leaf vertices of degree $2$ and using \cite
[Theorem~2.5]{grone et al.}, the obtained graph is two cycles like
graph in the case $1$. So $m_G(2)=3$ and the proof is complete. }
\end{proof}


\begin{example}
$G_1$ and $G_2$ are two unicyclic graphs with
$m_{G_1}(2)=m_{G_2}(2)=2$. Also, there exists at least one $k$ such
that $|T^{(i)}_k|\geq 3$ for $i=1$ or i=$2$. The bicyclic graph
$G_1\odot_{uv}G_2$ has $2$ among its Laplacian eigenvalues with
multiplicity $3$.

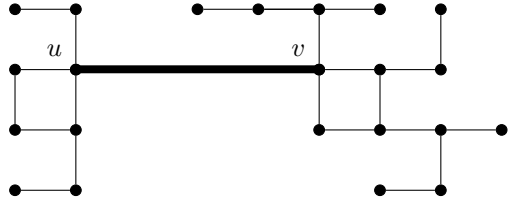
\begin{figure}[!ht]
  \[\begin{tikzpicture}[scale=.4,thin]
 \vertex[fill] (v3) at (-5,1.4)[label=north:$$] {};
 \vertex[fill] (v3) at (-5,1.4) [label=north west:$u$] {};
 \vertex[fill] (v4) at (-7,1.4)[label=north:$$] {};
 \vertex[fill] (v5) at (-5,-0.6)[label=south:$$] {};
 \vertex[fill] (v6) at (-7,-0.6)[label=south:$$] {};
 \vertex[fill](v7) at (3,1.4) [label=north:$$] {};
 \vertex[fill](v7) at (3,1.4) [label=north west:$v$] {};
 \vertex[fill](v8) at (5,1.4) [label=north:$$] {};
 \vertex[fill](v9) at (7,1.4) [label=north:$$] {};
 \vertex[fill](v10) at (5,-0.6) [label=north:$$] {};
 \vertex[fill](v11) at (7,-0.6) [label=north:$$] {};
 \vertex[fill] (v1) at (-1,3.4) [label=north:$$] {};
 \vertex[fill](v12) at (3,3.4) [label=north:$$] {};
 \vertex[fill](v13) at (-7,3.4) [label=north:$$] {};
 \vertex[fill](v14) at (-5,3.4) [label=north:$$] {};
 \vertex[fill](v15) at (-7,-2.6) [label=north:$$] {};
 \vertex[fill](v16) at (-5,-2.6) [label=north:$$] {};
 \vertex[fill](v17) at (3,-0.6) [label=north:$$] {};
 \vertex[fill](v18) at (7,3.4) [label=north:$$] {};
 \vertex[fill](v19) at (1,3.4) [label=north:$$] {};
 \vertex[fill](v20) at (7,-2.6) [label=south:$$] {};
 \vertex[fill](v21) at (5,-2.6) [label=south:$$] {};
 \vertex[fill](v22) at (9,-0.6) [label=south:$$] {};
\vertex[fill](v23) at (5,3.4) [label=south:$$] {};
 \path
 (v19) edge (v1)
 (v19) edge (v12)
 (v3) edge (v4)
 (v4) edge (v6)
 (v3) edge (v5)
 (v5) edge (v6)
 (v5) edge (v16)
 (v16) edge (v15)
 (v3) edge (v14)
 (v13) edge (v14)
 (v12) edge (v19)
 (v12) edge (v23)
 (v12) edge (v7)
 (v8) edge (v7)
 (v17) edge (v7)
 (v8) edge (v9)
 (v18) edge (v9)
 (v11) edge (v10)
 (v10) edge (v8)
 (v10) edge (v17)
 (v11) edge (v22)
 (v11) edge (v20)
 (v20) edge (v21);
 \draw[line width=3pt] (-5,1.41) .. controls (-1,1.41) .. (3,1.41);
\end{tikzpicture} \]
\caption{$G=G_1\odot_{uv} G_2$ with $m_{G}(2)=3$}
\end{figure}
\end{example}

The following example shows that the converse of Theorem \ref{multi}
does not necessarily true in general.

\begin{example}
Let $G$ be a unicyclic graph same as \emph{Figure $2$}. $G$ has $2$
among its Laplacian eigenvalues with multiplicity $1$. But
$G'=G\odot_{uv} G$ has $2$ among its Laplacian eigenvalues with
multiplicity $3$ (\emph{Figure $3$}).

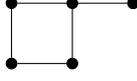
\begin{figure}[!ht]
  \[\begin{tikzpicture}[scale=.4,thin]
 \vertex[fill] (v2) at (-3,1.4) [label=south:$$] {};
 \vertex[fill] (v2) at (-3,1.4) [label=north:$$] {};
 \vertex[fill] (v3) at (-5,1.4)[label=north:$$] {};
 \vertex[fill] (v4) at (-7,1.4)[label=north:$$] {};
 \vertex[fill] (v5) at (-5,-0.6)[label=south:$$] {};
 \vertex[fill] (v6) at (-7,-0.6)[label=south:$$] {};

 \path
 (v2) edge (v3)
 (v3) edge (v4)
 (v4) edge (v6)
 (v3) edge (v5)
 (v5) edge (v6);
\end{tikzpicture} \]
\caption{$G$ with $m_G(2)=1$}
\end{figure}

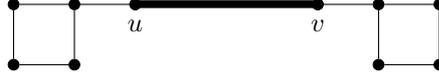
\begin{figure}[!ht]
  \[\begin{tikzpicture}[scale=.4,thin]
 \vertex[fill] (v2) at (-3,1.4) [label=south:$u$] {};
 \vertex[fill] (v2) at (-3,1.4) [label=north:$$] {};
 \vertex[fill] (v3) at (-5,1.4)[label=north:$$] {};
 \vertex[fill] (v4) at (-7,1.4)[label=north:$$] {};
 \vertex[fill] (v5) at (-5,-0.6)[label=south:$$] {};
 \vertex[fill] (v6) at (-7,-0.6)[label=south:$$] {};
 \vertex[fill](v7) at (3,1.4) [label=north:$$] {};
 \vertex[fill](v7) at (3,1.4) [label=south:$v$] {};
 \vertex[fill](v8) at (5,1.4) [label=north:$$] {};
 \vertex[fill](v9) at (7,1.4) [label=north:$$] {};
 \vertex[fill](v10) at (5,-0.6) [label=north:$$] {};
 \vertex[fill](v11) at (7,-0.6) [label=north:$$] {};

 \path
 (v2) edge (v3)
 (v3) edge (v4)
 (v4) edge (v6)
 (v3) edge (v5)
 (v5) edge (v6)
 (v8) edge (v7)
 (v8) edge (v9)
 (v11) edge (v9)
 (v11) edge (v10)
 (v10) edge (v8);
 \draw[line width=3pt] (-3,1.41) .. controls (0,1.41) .. (3,1.41);
\end{tikzpicture} \]
\caption{$G'=G\odot_{uv} G$ with $m_{G'}(2)=3$}
\end{figure}
\end{example}

 In \cite{mojdeh et al.} authors have considered a necessary and
sufficient condition in the bicyclic graph $G=G_1\odot_{uv} G_2$ for
having the Laplacian eigenvalue $2$, where $G_1$ and $G_2$ are
unicyclic graphs and have $2$ among their Laplacian eigenvalues. We
study the Laplacian eigenvalue $2$ of $G=G_1\odot_{uv} G_2$, where
$G_1$ and $G_2$ are unicyclic graphs and $G_1$ or $G_2$ dose not
have $2$ among its Laplacian eigenvalues by using Equation
(\ref{one}) without having the Laplacian characteristic polynomial.

\begin{thm}\label{laplac 2}
Let $G_1=C(T_1,\ldots ,T_{g_1})$ and $G_2=C(T_1,\ldots ,T_{g_2})$ be
unicyclic graphs such that $G_1$ has a perfect matching and $2$
among its Laplacian eigenvalues. Also, $G_2$ dose not have $2$ among
its Laplacian eigenvalues. Then $G=G_1\odot_{uv} G_2$ has $2$ among
its Laplacian eigenvalues if and only if $x(u)=0$.
\end{thm}

\begin{proof}{
 Let $G_1$ be a unicyclic graph containing a perfect matching such
that has $2$ among its Laplacian eigenvalues. So $s\equiv 0\
(mod\,4)$, where $s$ is the number of trees of odd orders in $G_1$,
by \cite [Theorem~9]{akbari et al.}.

\textbf{Case $1$.} If $g_1\not\equiv 0\ (mod\,4)$ or $g_1\equiv 0\
(mod\,4)$ and $s\neq g$ then $m_{G_1}(2)=1$ by \cite
[Theorem~13]{akbari et al.}
 and there exists the eigenvector corresponding to the
eigenvalue $2$ like $Y=(x_1,\ldots ,x_{n_1})^t$, such that $x_i\in
\{-1,1\}$ for $ i=1\cdots n_1$, by \cite [Theorem~4]{mojdeh et al.},
so $x(u)\neq 0$. By contrary, if $G$ has $2$ among its Laplacian
eigenvalues, then we can assume that $X=(x_1,\ldots ,x_{n_1+n_2})^t$
is an eigenvector of $L(G)$ corresponding to the eigenvalue $2$. All
vertices of $G$ satisfy in Equation (\ref{one}). Therefore
\begin{align*}
&\qquad \qquad \quad (d_G(u)-2)x(u)=\sum_{v_{j}\in N_G(u)} x(v_j)\\
& \Rightarrow (d_{G_1}(u)-2)x(u)+x(u)=\sum_{v_{j}\in N_{G_1}(u)} x(v_j)+x(v).\
\end{align*}
On the other hand, $G_1$ has $2$ among its Laplacian eigenvalues.
So, we have $x(v)=x(u)\neq 0$. Also
\begin{align*}
&\qquad  (d_G(v)-2)x(v)=\sum_{v_{j}\in N_G(v)} x(v_j)\\
&\Rightarrow (d_{G_2}(v)-2)x(v)+x(v)=\sum_{v_{j}\in N_{G_2}(v)} x(v_j)+x(u)\\
& \Rightarrow (d_{G_2}(v)-2)x(v)=\sum_{v_{j}\in N_{G_2}(v)}
x(v_j).\\
\end{align*}
Thus, by noting the fact that $d_{G_2}(w)=d_G(w)$ for the other
vertices of $G_2$, we have
$$(d_{G_2}(w)-2)x_{G_2}(w)=\sum_{{v_{i}}\in
N_{G_2}(w)}x(v_{i})\qquad \forall\, w\in V(G_2).$$ Hence, there
exist an eigenvector opposed to zero of $L(G_2)$ corresponding to
the eigenvalue $2$ and this is a contraction. Therefore, the proof
is complete.

\textbf{Case $2$.} If $s=g_1$ and $g_1\equiv 0\ (mod\,4)$ then
$m_{G_1}(2)=2$ by \cite [Theorem~13]{akbari et al.} and there exists
an eigenvector of $L(G_1)$ like $X=(x_1,\ldots ,x_{n_1})^t$
corresponding to the eigenvalue $2$, such that $x_i\in \{-1,0,1\}$.
For proving the ``if part'', let $x(u)=0$. Then by assigning $0$ to
all vertices of $G_2$, one can easily check that $Y=(X,0)^t$ is an
eigenvector of $L(G)$ corresponding to the eigenvalue $2$ (all
vertices of $G$ satisfy in Equation (\ref{one})). Conversely,
suppose $Y$ is an eigenvector of $L(G)$ corresponding to the
eigenvalue $2$ such that $y(u)\neq 0$. By a similar method in case
$1$, an eigenvector like $X\neq 0$ exists such that satisfy in
Equation (\ref{one}) for $\lambda=2$ in $L(G_2)$ and this is a
contradiction. Therefore, the result follows.}
\end{proof}
\begin{thm}\label{sun graph}
Let $G_1$ be a broken sun graph of order $n_1$ which has no perfect
matching and has $2$ among its Laplacian eigenvalues. Also, $G_2$ is
a unicyclic graph of order $n_2$ that does not have $2$ among its
Laplacian eigenvalues. So, $G=G_1\odot_{uv} G_2$ has $2$ among its
Laplacian eigenvalues if and only if $x(u)=0$.
\end{thm}

\begin{proof}{
First, assume that $g_1\equiv 0  \ (mod\,4)$ and there exist odd
number of vertices of degree 2 between any pair of consecutive
vertices of degree $3$, by \cite [Theorem~10]{akbari et al.}. We may
assign $\{-1, 0, 1\}$ to the vertices of $C_{g_1}$, by the pattern
$0, 1, 0,-1$ consecutively starting with a vertex of degree $3$, and
assign to each pendent vertex the negative of value of its neighbor,
to obtain an eigenvector of $L(G)$ corresponding to the eigenvalue
$2$ like $X=(x_1,\ldots,x_{n_1})^t$. Now, if $x(u)=0$ then by
assigning $0$ to all vertices of $G_2$, one can easily check that
$Y=(X,0)^t$ is an eigenvector of $L(G)$ corresponding to the
eigenvalue $2$ (all vertices of $G$ satisfy in Equation
(\ref{one})). Conversely, suppose $X$ is an eigenvector of $L(G)$
corresponding to the eigenvalue $2$ such that $x(u)\neq 0$. By a
similar method in Theorem \ref{laplac 2} case $1$, there exists an
eigenvector like $Z\neq 0$ such that $Z$ satisfy in Equation
(\ref{one}) for $\lambda=2$ in $L(G_2)$ and this is a contradiction.
Therefore, the proof is complete.}
\end{proof}

\begin{example}\label{one eigenvalue}
In the following \emph{Figure $(4)$}, $G_1$ and $G_2$ are broken
graphs such that just one of them has $2$ among its Laplacian
eigenvalues. By Theorem \ref{sun graph} and by Equation (\ref{one}),
$$X=(-1,0,0,1,0,0,0,0,0,0,0)^t$$ is an eigenvector of $L(G)$
corresponding to the eigenvalue $2$.
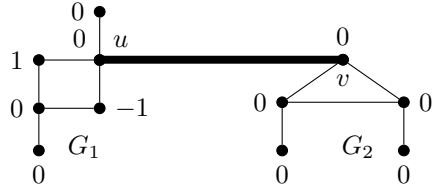
\begin{figure}[!ht]\label{3}
  \[\begin{tikzpicture}[scale=.4,thin]
 \vertex[fill] (v1) at (-5,1.4)[label=north west:$0$] {};
 \vertex[fill] (v1) at (-5,1.4)[label=north east:$u$] {};
 \vertex[fill] (v2) at (-7,1.4)[label=left:$1$] {};
 \vertex[fill] (v3) at (-5,3.0)[label=left:$0$] {};
 \vertex[fill] (v4) at (-5,-0.2)[label=right:$-1$] {};
 \vertex[fill] (v5) at (-7,-1.6)[label=south:$0$] {};
 \vertex[fill] (v6) at (-7,-0.2)[label=left:$0$] {};
 \vertex[fill] (v8) at (5,0.) [label=right:$0$] {};
 \vertex[fill] (v9) at (5,-1.6) [label=south:$0$] {};
 \vertex[fill](v10) at (3,1.41) [label=north:$0$] {};
 \vertex[fill](v10) at (3,1.41) [label=south:$v$] {};
 \vertex[fill](v11) at (1,0)[label=left:$0$]{};
  \vertex[fill](v12) at (1,-1.6)[label=south:$0$]{};
 \path
 (v1) edge (v2)
 (v1) edge (v3)
 (v2) edge (v6)
 (v4) edge (v6)
 (v6) edge (v5)
 (v1) edge (v4)
 (v1) edge (v10)
 (v10) edge (v8)
 (v10) edge (v11)
 (v8) edge (v11)
 (v8) edge (v9)
 (v11) edge (v12)
 ;

 \draw[line width=3pt] (-5,1.41) .. controls (-1,1.41) .. (3,1.41);
\draw (-3.5,-.4) ++(-2,-.4) node[anchor=north]{$G_1$} (-.5,-.4); \draw
(1,-.4) ++(2.5,-.4) node[anchor=north]{$G_2$} (4,-.4);
\end{tikzpicture} \]
\caption{$G=G_1\odot G_2$ has $2$ as its Laplacian eigenvalue. }
\end{figure}
\end{example}

In the follow, we show that two broken sun graphs $G_1$ and $G_2$ do
not have $2$ among their Laplacian eigenvalues, but $G=G_1\odot_{uv}
G_2$ may have $2$ among Laplacian eigenvalues.

\begin{thm}\label{eigenvector}
Let $G_1\in \mathfrak{U}_{n_1,g_1} $ and $G_2\in
\mathfrak{U}_{n_2,g_2}$, and $m_1$, $m_2$ be the numbers of leaf
vertices in $G_1$ and $G_2$ for which $m_1=g_1-1$ and $m_2=g_2-1$.
Let $d_{G_1}(u)=d_{G_2}(v)=2$. Then $L(G=G_1\odot_{uv} G_2)$ has $2$
among its Laplacian eigenvalues.
\end{thm}

\begin{proof}{
Let $G_1\in \mathfrak{U}_{n_1,g_1} $ and $G_2\in
\mathfrak{U}_{n_2,g_2}$. We may assign $1$ and $-1$ to the vertices
of $C_{g_1}$ and $C_{g_2}$, and assign $-1$ and $1$ to each leaf
vertices of $G_1$ and $G_2$, respectively. One may check that every
vertices of $G$ satisfies in Equation (\ref{one}). Therefore
$X=(x_1, \ldots ,x_{n_1+n_2})$ is an eigenvector  corresponding to
the Laplacian eigenvalue $2$, such that $x_i\in \{-1,1\} $,
\emph{for} $i=1,\ldots, n_1+n_2$ and we are done.  }
\end{proof}
\begin{example}
Let $G_1$ and $G_2$ be broken sun graphs (see Figure $5$), then
$G=G_1\odot_{uv} G_2$ has $2$ among its Laplacian eigenvalues by
Theorem \ref{eigenvector}.

\begin{figure}[!ht]\label{4}
  \[\begin{tikzpicture}[scale=.4,thin]
 \vertex[fill] (v1) at (-5,1.4)[label=north:$1$] {};
 \vertex[fill] (v1) at (-5,1.4)[label=south east:$u$] {};
 \vertex[fill] (v2) at (-7,1.4)[label=left:$1$] {};
 \vertex[fill] (v3) at (-7,3.4)[label=left:$-1$] {};
 \vertex[fill] (v4) at (-5,-0.6)[label=right:$1$] {};
 \vertex[fill] (v5) at (-5,-2.6)[label=south:$-1$] {};
 \vertex[fill] (v6) at (-7,-0.6)[label=left:$1$] {};
  \vertex[fill] (v7) at (-7,-2.6)[label=south:$-1$] {};
 \vertex[fill] (v8) at (5,0) [label=right:$-1$] {};
 \vertex[fill] (v9) at (5,-2) [label=south:$1$] {};
 \vertex[fill](v10) at (3,1.41) [label=north:$-1$] {};
 \vertex[fill](v10) at (3,1.41) [label=south:$v$] {};
 \vertex[fill](v11) at (1,0)[label=left:$-1$]{};
  \vertex[fill](v12) at (1,-2)[label=south:$1$]{};
 \path
 (v1) edge (v2)
 (v2) edge (v3)
 (v2) edge (v6)
 (v4) edge (v5)
 (v6) edge (v7)
 (v1) edge (v4)
 (v4) edge (v6)
 (v1) edge (v10)
 (v10) edge (v8)
 (v10) edge (v11)
 (v8) edge (v11)
 (v8) edge (v9)
 (v11) edge (v12)
 ;

 \draw[line width=3pt] (-5,1.41) .. controls (-1,1.41) .. (3,1.41);
\draw (-3.5,-2) ++(-2,-2) node[anchor=north]{$G_1$} (-.5,-2); \draw
(1,-2) ++(2.5,-2) node[anchor=north]{$G_2$} (4,-2);
\end{tikzpicture} \]
\caption{$G=G_1\odot_{uv} G_2$,$\lambda=2$ is an eigenvalue of
$L(G)$.}
\end{figure}
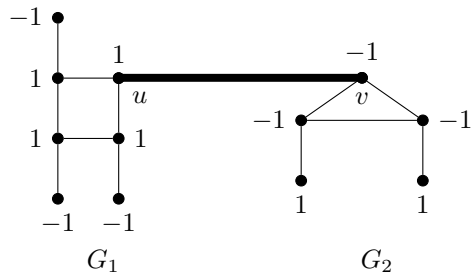
\end{example}


 Now, we finish this article with
extend this result for graphs with $n$ cycles.

\begin{lem} Let $G$ be a graph with $n$ cycles and
$\lambda> 1$ be an integral Laplacian eigenvalue of $G$. Therefore,
$m_{G}(\lambda)\leq n+1$.
\end{lem}

\begin{proof}{
If $G$ be a unicyclic graph then $m_{G}(\lambda)\leq 2$, by \cite
[Lemma~4] {akbari et al.}. By induction on $n$, let
$m_{G}(\lambda)\leq n$ for all graphs with cycles less than $n$. The
graph $G'=G-e$, where $e$ belongs to one of the cycles of $G$, is a
graph with $n-1$ cycles and by using the induction assumption
$m_{G'}(\lambda)\leq n$. If $m_{G}(\lambda)> n+1$ then
$m_{G'}(\lambda)> n$, by \cite [Theorem~3.2]{mohar} and this
contradicts. So, the result follows.}
\end{proof}


    \bibliographystyle{siamplain}

    \bibliography{references}

\end{document}